\documentclass[11pt]{amsart}

\usepackage[T1]{fontenc}
\usepackage[utf8]{inputenc}
\usepackage{lmodern}
\usepackage{amsmath,amssymb,amsthm,mathtools}
\usepackage{microtype}
\usepackage{xcolor}
\usepackage[hidelinks]{hyperref}
\usepackage{enumitem}
\usepackage{aliascnt}

\theoremstyle{plain}
\newtheorem{theorem}{Theorem}[section]

\newaliascnt{proposition}{theorem}
\newtheorem{proposition}[proposition]{Proposition}
\aliascntresetthe{proposition}

\newaliascnt{lemma}{theorem}
\newtheorem{lemma}[lemma]{Lemma}
\aliascntresetthe{lemma}

\newaliascnt{corollary}{theorem}
\newtheorem{corollary}[corollary]{Corollary}
\aliascntresetthe{corollary}

\theoremstyle{definition}
\newaliascnt{definition}{theorem}
\newtheorem{definition}[definition]{Definition}
\aliascntresetthe{definition}

\newaliascnt{example}{theorem}
\newtheorem{example}[example]{Example}
\aliascntresetthe{example}

\theoremstyle{remark}
\newaliascnt{remark}{theorem}

\aliascntresetthe{remark}

\usepackage[nameinlink,noabbrev]{cleveref}

\crefname{theorem}{Theorem}{Theorems}
\Crefname{theorem}{Theorem}{Theorems}
\crefname{proposition}{Proposition}{Propositions}
\Crefname{proposition}{Proposition}{Propositions}
\crefname{lemma}{Lemma}{Lemmas}
\Crefname{lemma}{Lemma}{Lemmas}
\crefname{corollary}{Corollary}{Corollaries}
\Crefname{corollary}{Corollary}{Corollaries}
\crefname{definition}{Definition}{Definitions}
\Crefname{definition}{Definition}{Definitions}
\crefname{example}{Example}{Examples}
\Crefname{example}{Example}{Examples}

\DeclareMathOperator{\M}{\mathcal M}
\DeclareMathOperator{\TV}{TV}

\newcommand{\R}{\mathbb R}
\newcommand{\N}{\mathbb N}
\newcommand{\Nzero}{\mathbb N_0}
\newcommand{\X}{X_S}
\newcommand{\Y}{Y_S}
\newcommand{\1}{\mathbf 1}
\newcommand{\dd}{\,d}
\newcommand{\norm}[1]{\left\|#1\right\|}

\title[Effective Intrinsic Ergodicity on \(S\)-gap shifts]{Effective Intrinsic Ergodicity for renewal-type potentials on \(S\)-gap shifts}

\author[F. Haydarov]{Farhod Haydarov}
\address{V. I. Romanovskiy Institute of Mathematics, Uzbekistan Academy of Sciences, Tashkent, Uzbekistan}
\email{haydarov\_imc@mail.ru}

\author[S. Kadyrov]{Shirali Kadyrov}
\address{Department of Mathematics and Natural Sciences, SDU University, Kaskelen, Kazakhstan}
\email{shirali.kadyrov@sdu.edu.kz}

\author[M. Makhmudov]{Mirmukhsin Makhmudov}
\address{Research Unit of Mathematical Sciences, University of Oulu, Oulu, Finland}
\email{Mirmukhsin.Makhmudov@oulu.fi}

\author[K. Mamayusupov]{Khudoyor Mamayusupov}
\address{New Uzbekistan University, Tashkent, Uzbekistan}
\email{k.mamayusupov@newuu.uz}

\date{}

\begin{document}

\begin{abstract}
We establish effective intrinsic ergodicity for renewal-type potentials on one-sided \(S\)-gap shifts. Inducing on the one-symbol cylinder \([1]\) reduces the system to a full shift over the alphabet \(S\), where the induced potential becomes a one-symbol potential and the equilibrium measure is Bernoulli. The associated renewal equation has a unique solution \(P\), and under the condition \(P>\phi(0^\infty)\) (automatic when \(S\) is infinite), we show that \(P\) is the topological pressure and that the potential admits a unique equilibrium state \(\mu_\phi\).

Our main result is an effective intrinsic ergodicity estimate: invariant measures whose free energy is within \(\Delta\) of the pressure are \(O(\sqrt{\Delta})\)-close to \(\mu_\phi\) when tested against Hölder observables. As an application, every finite-word cylinder of positive \(\mu_\phi\)-measure yields a uniform pressure gap for the set of orbits avoiding that cylinder, leading in the entropy case to strict entropy and Hausdorff-dimension gaps.
\end{abstract}

\maketitle

\section{Introduction}

The variational principle is one of the central facts in thermodynamic formalism.  For a continuous map \(T:X\to X\) on a compact metric space and a continuous potential \(\phi:X\to\R\), it identifies the topological pressure with
\[
P_{\mathrm{top}}(X,T,\phi)
=\sup\left\{h_\mu(T)+\int \phi\,\dd\mu:
\mu \text{ is }T\text{-invariant}\right\}.
\]
A measure attaining this supremum is called an equilibrium state.  Classical results of Bowen, Ruelle, Walters and others give uniqueness and statistical properties for many uniformly hyperbolic systems and shifts of finite type \cite{bowen1975equilibrium,ruelle2004thermodynamic,walters1975variational,walters1982,keller1998equilibrium}.  A parallel operator-theoretic approach, based on transfer operators and spectral stability, is also central in the classical theory \cite{baladi2000positive,keller1999stability,gouezel2006banach,baladi2014linear}.

A natural problem is to understand what remains true beyond the finite-state Markov setting.  For countable Markov shifts, Sarig developed a thermodynamic formalism that extends many features of the classical theory \cite{sarig1999thermodynamic,sarig2001thermodynamic}.  For compact symbolic systems beyond specification, Climenhaga and Thompson proved uniqueness results using orbit decompositions and pressure estimates \cite{climenhaga2012intrinsic,climenhaga2013bowen,climenhaga2016unique}.  In particular, \(S\)-gap shifts are a standard family of compact symbolic systems which need not be shifts of finite type.

This paper concerns a quantitative version of uniqueness.  Suppose that \(\mu_\phi\) is the unique equilibrium state.  If another invariant measure \(\mu\) has free energy close to the pressure, must \(\mu\) be close to \(\mu_\phi\)?  Effective versions of this question go back to work of Kadyrov for measures of maximal entropy on finite-state shifts \cite{kadyrov2015effective}.  Related estimates for countable Markov shifts were proved by R\"uhr and by R\"uhr--Sarig \cite{ruhr2021pressure,ruhr2022effective}.  Recently, Pollicott gave an analogue for expanding interval maps, including beta-transformations \cite{pollicott2026}.  The aim here is to prove such an estimate for a compact symbolic system which is naturally coded by a countable renewal shift.

Let \(S\subset \Nzero\) be nonempty.  The one-sided \(S\)-gap shift \(\X\subset\{0,1\}^{\Nzero}\) consists of all sequences for which the number of zeros between two successive symbols \(1\) belongs to \(S\).  No mixing assumption is needed in the main theorem.  The cylinder
\[
E=[1]=\{x\in\X:x_0=1\}
\]
has a simple first-return structure.  If \(x\in E\) returns to \(E\), then the return time is \(n+1\) for some \(n\in S\), and the block read by the orbit is \(10^n\).  Thus inducing on \(E\) gives the full shift
\[
\Y=S^{\Nzero}
\]
with return-time function \(\tau(y)=y_0+1\).

We study H\"older potentials \(\phi\) on \(\X\) whose induced potential depends only on the current return block.  More precisely, if
\[
\Phi_\phi(y):=\sum_{j=0}^{\tau(y)-1}\phi(\sigma^j\pi(y)),
\]
where \(\pi(y)=10^{y_0}10^{y_1}10^{y_2}\cdots\), then \(\phi\) is called renewal-type if
\[
\Phi_\phi(y)=\psi(y_0)
\]
for some function \(\psi:S\to\R\).  The induced equilibrium problem is then one-dimensional.  The pressure parameter \(P\) is determined by the renewal equation
\begin{equation}\label{eq:intro-renewal-equation}
\sum_{n\in S}e^{\psi(n)-P(n+1)}=1.
\end{equation}
We prove below that this equation has a unique real solution.  The induced equilibrium measure is the Bernoulli measure with weights
\[
p_n=e^{\psi(n)-P(n+1)}.
\]

The passage from the induced system back to \(\X\) is not formal.  First, a bounded H\"older observable \(f\) on \(\X\) induces to
\[
F_f(y)=\sum_{j=0}^{\tau(y)-1}f(\sigma^j\pi(y)),
\]
which is generally unbounded because \(\tau\) is unbounded.  We handle this by writing
\[
F_f=f(0^\infty)\tau+G_f,
\]
where \(G_f\) is bounded and H\"older.  Second, the inducing formulas do not apply to an arbitrary invariant measure as soon as \(\mu([1])>0\).  The reason is that
\[
\bigcup_{j\ge 0}\sigma^{-j}[1]=\X\setminus\{0^\infty\}.
\]
Thus a measure may have both a positive atom at \(0^\infty\) and positive mass on \([1]\).  Before inducing, one must split off the atom at \(0^\infty\).  This correction is built into the proof below.

We now state the main theorem in a form that does not require topological mixing.  Fix \(\theta\in(0,1)\).  For \(x,y\in\X\), set
\[
d_\theta(x,y)=\theta^{t(x,y)},
\qquad
 t(x,y)=\min\{k\ge 0:x_k\ne y_k\},
\]
with the convention \(d_\theta(x,x)=0\).  Let \(C^\theta(\X)\) be the Banach space of bounded H\"older functions with norm
\[
\norm{f}_\theta=\norm{f}_\infty+|f|_\theta,
\qquad
|f|_\theta=\sup_{x\ne y}\frac{|f(x)-f(y)|}{d_\theta(x,y)}.
\]

\begin{theorem}\label{thm:main}
Let \(S\subset\Nzero\) be nonempty, and let \(\X\) be the one-sided \(S\)-gap shift.  Let \(\phi\in C^\theta(\X)\) be a renewal-type potential with induced one-symbol potential \(\psi:S\to\R\).  Let \(P\) be the unique real solution of \eqref{eq:intro-renewal-equation}.  If \(S\) is finite, assume the pressure separation condition
\[
P>\phi(0^\infty).
\]
If \(S\) is infinite, no additional pressure assumption is needed. Under these hypotheses,
\[
P=P_{\mathrm{top}}(\X,\sigma,\phi),
\]
and \(\phi\) has a unique equilibrium state \(\mu_\phi\).  Moreover, there exists a constant \(C>0\), depending only on \(S,\phi\) and \(\theta\), such that for every \(\sigma\)-invariant Borel probability measure \(\mu\) on \(\X\) and every \(f\in C^\theta(\X)\),
\[
\left|\int f\,\dd\mu-\int f\,\dd\mu_\phi\right|
\le
C\norm{f}_\theta
\sqrt{P-\left(h_\mu(\sigma)+\int\phi\,\dd\mu\right)}.
\]
\end{theorem}

The estimate in \cref{thm:main} has the same square-root form as the inequalities appearing in effective intrinsic ergodicity.  For shifts of finite type and the measure of maximal entropy, such estimates were proved by Kadyrov \cite{kadyrov2015effective}.  R\"uhr and R\"uhr--Sarig developed related estimates in countable-state settings \cite{ruhr2021pressure,ruhr2022effective}.  Pollicott's recent note proves an analogue for expanding interval maps and beta-transformations \cite{pollicott2026}.  The present paper differs from those results in that the system \(\X\) is compact, but the natural proof passes through a countable induced renewal shift and then returns to \(\X\).

The paper is organized as follows.  \Cref{sec:inducing} introduces \(S\)-gap shifts, the renewal coding, the lift from the induced shift, and the atom-splitting lemma.  \Cref{sec:renewal-potentials} defines renewal-type potentials and proves the exponential tail estimate used later.  \Cref{sec:equilibrium} constructs the equilibrium state and proves uniqueness.  \Cref{sec:induced-gap} proves the square-root estimate on the induced renewal shift.  \Cref{sec:transfer} transfers the estimate back to \(\X\) and proves \cref{thm:main}.  \Cref{sec:applications} gives the cylinder-avoidance application.

\section{The renewal coding and the inducing formulas}\label{sec:inducing}

\subsection{The \texorpdfstring{\(S\)}{S}-gap shift}

Let \(S\subset \Nzero\) be nonempty.  The one-sided \(S\)-gap shift \(\X\) is the set of all \(x=(x_k)_{k\ge0}\in\{0,1\}^{\Nzero}\) such that whenever two consecutive symbols \(1\) occur in \(x\), the number of zeros between them belongs to \(S\).  The point
\[
0^\infty=000\cdots
\]
always belongs to \(\X\).  We write \(\sigma:\X\to\X\) for the left shift and \(\M_\sigma(\X)\) for the set of \(\sigma\)-invariant Borel probability measures on \(\X\).

For reference, topological mixing of an \(S\)-gap shift is characterized by
\[
\gcd\{n+1:n\in S\}=1,
\]
with the usual non-degeneracy assumptions; see, for example, \cite{lindmarcus2021,dastjerdi2012}.  This characterization is not used below; the theorem holds without assuming mixing.

Let
\[
E=[1]=\{x\in\X:x_0=1\}.
\]
Then
\begin{equation}\label{eq:sweep}
\bigcup_{j\ge0}\sigma^{-j}E=\X\setminus\{0^\infty\}.
\end{equation}
Indeed, a point belongs to the union on the left precisely when it has at least one symbol \(1\).

\subsection{The induced renewal shift}

Let
\[
\Y=S^{\Nzero}
\]
with the left shift, again denoted by \(\sigma\).  Define
\[
\pi:\Y\to E,
\qquad
\pi(y)=10^{y_0}10^{y_1}10^{y_2}\cdots,
\]
and define the return-time function
\[
\tau(y)=y_0+1.
\]
On points of \(E\) which return to \(E\) infinitely often, the first return map to \(E\) is conjugate to \((\Y,\sigma)\) through \(\pi\).

If \(\nu\in\M_\sigma(\Y)\) and
\[
m_\nu:=\int \tau\,\dd\nu<\infty,
\]
we define its lift \(L(\nu)\in\M_\sigma(\X)\) by
\begin{equation}\label{eq:lift}
\int f\,\dd L(\nu)
=\frac1{m_\nu}
\int \sum_{j=0}^{\tau(y)-1}f(\sigma^j\pi(y))\,\dd\nu(y)
\end{equation}
for every bounded measurable \(f:\X\to\R\).  This is the standard tower construction.  It gives a \(\sigma\)-invariant probability measure on \(\X\), and it gives no mass to \(0^\infty\).

\subsection{Inducing measures on \texorpdfstring{\([1]\)}{[1]}}

The following formulation is the one used in the paper.  The condition \(\mu(\{0^\infty\})=0\) is essential.

\begin{lemma}[Kac--Abramov formulas in the present setting]\label{lem:kac-abramov}
Let \(\mu\in\M_\sigma(\X)\) satisfy \(\mu(\{0^\infty\})=0\).  Then \(\mu(E)>0\).  Let \(\nu\in\M_\sigma(\Y)\) be the measure obtained by normalizing \(\mu|_E\) and transporting it to \(\Y\) through \(\pi^{-1}\).  Then
\[
m:=\int\tau\,\dd\nu=\frac1{\mu(E)},
\]
\[
h_\mu(\sigma)=\frac{h_\nu(\sigma)}m,
\]
and, for every bounded measurable \(f:\X\to\R\),
\begin{equation}\label{eq:induced-integral}
\int f\,\dd\mu
=\frac1m\int F_f\,\dd\nu,
\qquad
F_f(y)=\sum_{j=0}^{\tau(y)-1}f(\sigma^j\pi(y)).
\end{equation}
Conversely, if \(\nu\in\M_\sigma(\Y)\) and \(\int\tau\,\dd\nu<\infty\), then \(L(\nu)\) satisfies these formulas with \(m=\int\tau\,\dd\nu\).
\end{lemma}

\begin{proof}
Since \(\mu(\{0^\infty\})=0\), \eqref{eq:sweep} gives
\[
1=\mu\left(\bigcup_{j\ge0}\sigma^{-j}E\right).
\]
If \(\mu(E)=0\), invariance gives \(\mu(\sigma^{-j}E)=0\) for every \(j\), a contradiction.  Hence \(\mu(E)>0\).  Poincar\'e recurrence implies that the normalized measure on \(E\) is supported on points returning to \(E\) infinitely often. Equivalently, the points with only finitely many symbols \(1\) have zero \(\mu\)-measure, since they eventually enter \(0^\infty\) and \(\mu(\{0^\infty\})=0\).  The first return map is then identified with \((\Y,\sigma)\) by \(\pi\).  Kac's formula, Abramov's entropy formula and the tower integral identity give the displayed formulas; see \cite{kac1947,abramov1959,walters1982,aaronson1997}.  The converse is the standard tower construction.\end{proof}

\subsection{Splitting off the fixed point}

For arbitrary invariant measures one must first remove the possible atom at \(0^\infty\).

\begin{lemma}[Atom splitting]\label{lem:atom-splitting}
Let \(\mu\in\M_\sigma(\X)\), and set
\[
a:=\mu(\{0^\infty\}),
\qquad
\beta:=1-a.
\]
If \(\beta>0\), define
\[
\widetilde\mu:=\frac{\mu-a\delta_{0^\infty}}{\beta}.
\]
Then \(\widetilde\mu\in\M_\sigma(\X)\), \(\widetilde\mu(\{0^\infty\})=0\), and
\[
\mu=a\delta_{0^\infty}+\beta\widetilde\mu.
\]
Moreover, for every continuous potential \(\phi\),
\[
h_\mu(\sigma)+\int\phi\,\dd\mu
=
\beta\left(h_{\widetilde\mu}(\sigma)+\int\phi\,\dd\widetilde\mu\right)
+a\phi(0^\infty).
\]
Consequently, if \(P\in\R\) and
\[
\Delta_\phi(\eta):=P-\left(h_\eta(\sigma)+\int\phi\,\dd\eta\right),
\]
then
\begin{equation}\label{eq:defect-splitting}
\Delta_\phi(\mu)=a\bigl(P-\phi(0^\infty)\bigr)+\beta\Delta_\phi(\widetilde\mu)
\end{equation}
whenever \(\beta>0\).
\end{lemma}

\begin{proof}
The point \(0^\infty\) is fixed, so removing its atom and renormalizing gives an invariant probability measure when \(\beta>0\).  The displayed decomposition of the free energy follows from the affinity of measure-theoretic entropy on invariant probability measures and from the linearity of the integral; see \cite{walters1982}.  The formula for the defect is obtained by subtracting the free energy from \(P\).\end{proof}

\section{Renewal-type potentials}\label{sec:renewal-potentials}

We use the same symbolic metric on \(\Y\):
\[
d_\theta(y,z)=\theta^{t(y,z)},
\qquad
 t(y,z)=\min\{k\ge0:y_k\ne z_k\},
\]
with \(d_\theta(y,y)=0\).  The space \(C^\theta(\Y)\) consists of bounded H\"older functions with the corresponding norm.

\begin{definition}\label{def:renewal-type}
Let \(\phi\in C^\theta(\X)\).  The induced potential of \(\phi\) is
\[
\Phi_\phi(y)=\sum_{j=0}^{\tau(y)-1}\phi(\sigma^j\pi(y)).
\]
We say that \(\phi\) is \emph{renewal-type} if there is a function \(\psi:S\to\R\) such that
\[
\Phi_\phi(y)=\psi(y_0)
\qquad (y\in\Y).
\]
In this case \(\psi\) is called the induced one-symbol potential.
\end{definition}

Equivalently, the total weight accumulated along a return block \(10^n\) depends only on \(n\).  The following examples illustrate the condition and also show that it is a genuine restriction.

\subsection{Examples and nonexamples}

\begin{example}\label{ex:renewal-type-examples}
The following potentials are typical.
\begin{enumerate}[label=(\roman*)]
\item The zero potential \(\phi\equiv0\) is renewal-type with \(\psi(n)=0\) for all \(n\in S\).

\item More generally, suppose that \(\phi\in C^\theta(\X)\) and that, for each \(n\in S\), there is a constant \(c_n\in\R\) such that
\[
\sum_{j=0}^{n}\phi(\sigma^j x)=c_n
\qquad\text{for every }x\in[10^n1].
\]
Then \(\Phi_\phi(y)=c_{y_0}\), so \(\phi\) is renewal-type with \(\psi(n)=c_n\).

\item Fix \(0<\rho<\theta\), and define
\[
\phi_\rho(x)=
\begin{cases}
\rho^n, & x\in[10^n1]\text{ for some }n\in S,\\
0, & \text{otherwise}.
\end{cases}
\]
Equivalently, \(\phi_\rho=\sum_{n\in S}\rho^n\1_{[10^n1]}\), since the cylinders \([10^n1]\) are disjoint.  If \(t(x,z)=m\ge1\) and the values of \(\phi_\rho\) at \(x\) and \(z\) differ, then any nonzero value involved is at most \(\rho^{m-1}\).  Hence
\[
|\phi_\rho(x)-\phi_\rho(z)|\le 2\rho^{m-1}\le \frac2\rho\,\theta^m,
\]
and \(\phi_\rho\in C^\theta(\X)\).  Along a first return block, only the initial point begins with \(1\), so
\[
\Phi_{\phi_\rho}(y)=\phi_\rho(\pi(y))=\rho^{y_0}.
\]
Thus \(\phi_\rho\) is renewal-type with \(\psi(n)=\rho^n\).

\item The renewal-type condition is not automatic.  Assume that \(0\in S\) and \(S\cap\N\ne\emptyset\).  Then \(\phi=\1_{[010]}\) is not renewal-type.  Indeed, for \(\pi(y)=10^{y_0}10^{y_1}\cdots\), the word \(010\) occurs during the first return block exactly when \(y_0\ge1\) and \(y_1\ge1\).  Thus
\[
\Phi_\phi(y)=\1_{\{y_0\ge1,\,y_1\ge1\}},
\]
which depends on \(y_1\) as well as \(y_0\).
\end{enumerate}
\end{example}

The next estimate is the source of the exponential tail of the induced equilibrium distribution.

\begin{lemma}\label{lem:psi-tail}
Let \(\phi\in C^\theta(\X)\) be renewal-type with induced potential \(\psi:S\to\R\).  Then, for every \(n\in S\),
\begin{equation}\label{eq:psi-tail}
\left|\psi(n)-(n+1)\phi(0^\infty)\right|
\le
\frac{|\phi|_\theta}{1-\theta}.
\end{equation}
Consequently, if \(P>\phi(0^\infty)\) and
\[
p_n=e^{\psi(n)-P(n+1)},
\]
then there is \(\eta>0\) such that
\begin{equation}\label{eq:p-exp-tail}
\sum_{n\in S}p_n e^{\eta(n+1)}<\infty.
\end{equation}
In particular, \(\sum_{n\in S}(n+1)p_n<\infty\).
\end{lemma}

\begin{proof}
Fix \(n\in S\) and choose \(y\in\Y\) with \(y_0=n\).  Since \(\phi\) is renewal-type,
\[
\psi(n)=\sum_{j=0}^{n}\phi(\sigma^j\pi(y)).
\]
For \(j=0\), the distance from \(\pi(y)\) to \(0^\infty\) is at most \(1\).  For \(1\le j\le n\), the point \(\sigma^j\pi(y)\) begins with \(n+1-j\) zeros before the next \(1\), hence
\[
d_\theta(\sigma^j\pi(y),0^\infty)\le \theta^{n+1-j}.
\]
Therefore
\[
\begin{aligned}
\left|\psi(n)-(n+1)\phi(0^\infty)\right|
&\le |\phi|_\theta\left(1+\sum_{j=1}^{n}\theta^{n+1-j}\right)  \\
&\le \frac{|\phi|_\theta}{1-\theta}.
\end{aligned}
\]
If \(\eta=(P-\phi(0^\infty))/2\), then \eqref{eq:psi-tail} gives
\[
p_n e^{\eta(n+1)}
\le
\exp\left(\frac{|\phi|_\theta}{1-\theta}\right)e^{-\eta(n+1)}.
\]
The right-hand side is summable over \(S\subset\Nzero\), proving \eqref{eq:p-exp-tail}.\end{proof}

\begin{lemma}[The renewal equation]\label{lem:renewal-solution}
Let \(\phi\in C^\theta(\X)\) be renewal-type with induced one-symbol potential \(\psi:S\to\R\).  The equation
\begin{equation}\label{eq:renewal-equation}
\sum_{n\in S}e^{\psi(n)-P(n+1)}=1
\end{equation}
has a unique real solution \(P\). If \(S\) is infinite, then this solution satisfies
\[
P>\phi(0^\infty).
\]
\end{lemma}

\begin{proof}
Set
\[
Z(t)=\sum_{n\in S}e^{\psi(n)-t(n+1)}.
\]
For finite \(S\), the function \(Z\) is continuous and strictly decreasing on \(\R\), with \(Z(t)\to\infty\) as \(t\to-\infty\) and \(Z(t)\to0\) as \(t\to\infty\). Hence there is a unique \(P\in\R\) with \(Z(P)=1\).

Assume now that \(S\) is infinite. Let \(B=|\phi|_\theta/(1-\theta)\). By \cref{lem:psi-tail},
\[
e^{-B}e^{-(t-\phi(0^\infty))(n+1)}
\le e^{\psi(n)-t(n+1)}
\le e^{B}e^{-(t-\phi(0^\infty))(n+1)}.
\]
Thus \(Z(t)<\infty\) for \(t>\phi(0^\infty)\), while \(Z(t)=\infty\) for \(t\le\phi(0^\infty)\). On \((\phi(0^\infty),\infty)\), the function \(Z\) is continuous and strictly decreasing. Also \(Z(t)\to\infty\) as \(t\downarrow\phi(0^\infty)\), and \(Z(t)\to0\) as \(t\to\infty\). Hence there is a unique solution of \(Z(P)=1\), and it satisfies \(P>\phi(0^\infty)\).\end{proof}

Let \(P\) denote the unique solution of \eqref{eq:renewal-equation}. Throughout the rest of the paper we assume the pressure separation condition
\begin{equation}\label{eq:pressure-separation}
P>\phi(0^\infty).
\end{equation}
By \cref{lem:renewal-solution}, this condition is automatic when \(S\) is infinite. Define
\begin{equation}\label{eq:p-def}
p_n=e^{\psi(n)-P(n+1)},
\qquad n\in S,
\end{equation}
and let \(\nu_\phi\) be the Bernoulli measure on \(\Y\) with one-coordinate distribution \(p=(p_n)_{n\in S}\).  By \cref{lem:psi-tail},
\[
\bar\tau_\phi:=\int \tau\,\dd\nu_\phi=\sum_{n\in S}(n+1)p_n<\infty.
\]
Thus the lifted measure
\begin{equation}\label{eq:mu-phi-lift}
\mu_\phi:=L(\nu_\phi)
\end{equation}
is well defined.

\section{The equilibrium state}\label{sec:equilibrium}

We begin with a countable-alphabet entropy estimate.  For a probability vector \(q=(q_n)_{n\in S}\), write
\[
H(q)=-\sum_{n\in S}q_n\log q_n
\]
with the convention \(0\log0=0\), and write
\[
D(q\|p)=\sum_{n\in S}q_n\log\frac{q_n}{p_n}
\]
for relative entropy.

\begin{lemma}\label{lem:finite-entropy}
Let \(q=(q_n)_{n\in S}\) be a probability vector such that
\[
\sum_{n\in S}q_n(n+1)<\infty.
\]
Then \(H(q)<\infty\), \(\sum_n q_n|\psi(n)|<\infty\), and
\begin{equation}\label{eq:relative-entropy-identity}
D(q\|p)
=
-H(q)-\sum_{n\in S}q_n\psi(n)
+P\sum_{n\in S}q_n(n+1).
\end{equation}
In particular,
\begin{equation}\label{eq:gibbs-inequality}
H(q)+\sum_{n\in S}q_n\psi(n)
\le
P\sum_{n\in S}q_n(n+1),
\end{equation}
with equality if and only if \(q=p\).
\end{lemma}

\begin{proof}
The finite first moment implies finite entropy by comparison with an exponential distribution.  Indeed, for any \(a>0\), the partition function \(Z_a=\sum_{n\in S}e^{-a(n+1)}\) is finite, and the Gibbs inequality gives
\[
H(q)\le a\sum_{n\in S}q_n(n+1)+\log Z_a<\infty.
\]
Since \(\phi\) is bounded and \(\psi(n)\) is a sum of \(n+1\) values of \(\phi\),
\[
|\psi(n)|\le (n+1)\norm{\phi}_\infty,
\]
so \(\sum_nq_n|\psi(n)|<\infty\).  Finally, \(\log p_n=\psi(n)-P(n+1)\), and therefore
\[
D(q\|p)=\sum_nq_n\log q_n-\sum_nq_n\log p_n,
\]
which is exactly \eqref{eq:relative-entropy-identity}.  The inequality and equality case are the standard properties of relative entropy.\end{proof}

\begin{proposition}\label{prop:equilibrium}
The measure \(\mu_\phi=L(\nu_\phi)\) is the unique equilibrium state of \(\phi\) on \(\X\), and
\[
P_{\mathrm{top}}(\X,\sigma,\phi)=P.
\]
\end{proposition}

\begin{proof}
First consider an invariant measure \(\eta\in\M_\sigma(\X)\) with \(\eta(\{0^\infty\})=0\).  Let \(\nu\in\M_\sigma(\Y)\) be its induced measure, and put
\[
m=\int\tau\,\dd\nu.
\]
By \cref{lem:kac-abramov},
\[
h_\eta(\sigma)+\int\phi\,\dd\eta
=\frac1m\left(h_\nu(\sigma)+\int\psi(y_0)\,\dd\nu\right).
\]
Let \(q\) be the one-coordinate marginal of \(\nu\).  Then \(\sum_nq_n(n+1)=m<\infty\).  Since entropy rate is bounded above by one-symbol entropy,
\[
h_\nu(\sigma)\le H(q),
\]
and \cref{lem:finite-entropy} gives
\[
H(q)+\int\psi(y_0)\,\dd\nu
\le Pm.
\]
Therefore
\begin{equation}\label{eq:free-energy-upper-no-atom}
h_\eta(\sigma)+\int\phi\,\dd\eta\le P.
\end{equation}

Now let \(\mu\in\M_\sigma(\X)\) be arbitrary.  If \(\mu=\delta_{0^\infty}\), then
\[
h_\mu(\sigma)+\int\phi\,\dd\mu=\phi(0^\infty)<P.
\]
Otherwise, apply \cref{lem:atom-splitting}.  Writing \(\mu=a\delta_{0^\infty}+\beta\widetilde\mu\), \(\beta>0\), we have
\[
\begin{aligned}
h_\mu(\sigma)+\int\phi\,\dd\mu
&=a\phi(0^\infty)+\beta\left(h_{\widetilde\mu}(\sigma)+\int\phi\,\dd\widetilde\mu\right)\\
&\le a\phi(0^\infty)+\beta P\le P.
\end{aligned}
\]
Thus every invariant measure has free energy at most \(P\).

For the lifted Bernoulli measure \(\mu_\phi=L(\nu_\phi)\), \cref{lem:kac-abramov} applies and
\[
h_{\nu_\phi}(\sigma)=H(p).
\]
Using \(\log p_n=\psi(n)-P(n+1)\), we get
\[
H(p)+\sum_{n\in S}p_n\psi(n)=P\sum_{n\in S}p_n(n+1)=P\bar\tau_\phi.
\]
Hence
\[
h_{\mu_\phi}(\sigma)+\int\phi\,\dd\mu_\phi=P.
\]
The variational principle now gives \(P_{\mathrm{top}}(\X,\sigma,\phi)=P\), and \(\mu_\phi\) is an equilibrium state.

It remains to prove uniqueness.  Let \(\mu\) be an equilibrium state, and set \(a=\mu(\{0^\infty\})\).  The atom-splitting formula and \(P>\phi(0^\infty)\) force \(a=0\), so \(\mu(\{0^\infty\})=0\).  Let \(\nu\) be the induced measure and let \(q\) be its one-coordinate marginal.  Equality in the preceding upper bound implies both
\[
q=p
\]
and
\[
h_\nu(\sigma)=H(q).
\]
We spell out the last consequence.  Since \(H(q)<\infty\), the entropy rate of the stationary coordinate process satisfies
\[
h_\nu(\sigma)=\lim_{k\to\infty}H_\nu(y_k\mid y_0,\ldots,y_{k-1}),
\]
and the conditional entropies form a non-increasing sequence bounded above by \(H(q)\); see \cite{walters1982,gray2011}.  If the limit is \(H(q)\), then every conditional entropy equals \(H(q)\).  Hence \(y_k\) is independent of \((y_0,\ldots,y_{k-1})\) for every \(k\), and all finite-dimensional distributions factor.  Thus \(\nu\) is the Bernoulli measure with marginal \(p\), i.e. \(\nu=\nu_\phi\).  Therefore \(\mu=L(\nu)=L(\nu_\phi)=\mu_\phi\).\end{proof}

\begin{corollary}\label{cor:block-weights}
For every \(n\in S\),
\[
\mu_\phi([10^n1])=\mu_\phi([1])e^{\psi(n)-P(n+1)}.
\]
Moreover,
\[
\mu_\phi([1])^{-1}=\sum_{n\in S}(n+1)e^{\psi(n)-P(n+1)}.
\]
\end{corollary}

\begin{proof}
By Kac's formula for \(\mu_\phi=L(\nu_\phi)\),
\[
\mu_\phi([1])^{-1}=\bar\tau_\phi.
\]
On \([1]\), the event that the first return block has gap \(n\) is exactly \([10^n1]\), and its induced measure is \(p_n\).\end{proof}

\section{The pressure-gap estimate on the induced shift}\label{sec:induced-gap}

For \(\nu\in\M_\sigma(\Y)\) with \(\int\tau\,\dd\nu<\infty\), define the induced pressure defect by
\begin{equation}\label{eq:induced-defect}
\Delta_Y(\nu)=P\int\tau\,\dd\nu-h_\nu(\sigma)-\int\psi(y_0)\,\dd\nu.
\end{equation}
By the proof of \cref{prop:equilibrium}, \(\Delta_Y(\nu)\ge0\).

We use the convention
\[
\TV(q,p)=\frac12\sum_{n\in S}|q_n-p_n|.
\]
Pinsker's inequality says
\[
\TV(q,p)\le \sqrt{\frac12D(q\|p)}.
\]

\begin{theorem}\label{thm:induced-gap}
For every \(H\in C^\theta(\Y)\) and every \(\nu\in\M_\sigma(\Y)\) with \(\int\tau\,\dd\nu<\infty\),
\[
\left|\int H\,\dd\nu-\int H\,\dd\nu_\phi\right|
\le
\frac{|H|_\theta}{\sqrt2(1-\theta)}\sqrt{\Delta_Y(\nu)}.
\]
\end{theorem}

\begin{proof}
For \(n\ge0\), define
\[
(T_nH)(y)=\int H(y_0,\ldots,y_{n-1},z_0,z_1,\ldots)\,\dd\nu_\phi(z),
\]
where for \(n=0\) the list \(y_0,\ldots,y_{n-1}\) is empty. Thus \(T_0H=\int H\,\dd\nu_\phi\). Since two points which agree in their first \(n\) coordinates are at distance at most \(\theta^n\), we have \(T_nH\to H\) uniformly. Hence
\[
\int H\,\dd\nu-\int H\,\dd\nu_\phi
=
\sum_{n=0}^\infty\int (T_{n+1}H-T_nH)\,\dd\nu,
\]
and the series is absolutely convergent.

Fix \(n\ge0\). Let \(\mathcal F_n\) be the sigma-algebra generated by \(y_0,\ldots,y_{n-1}\), with \(\mathcal F_0\) trivial. Let \(Q_n(\cdot\mid y_0,\ldots,y_{n-1})\) be a regular conditional distribution of \(y_n\) with respect to \(\mathcal F_n\). For a prefix \(a=(a_0,\ldots,a_{n-1})\) and \(b\in S\), put
\[
\Phi_n(a,b)=\int H(a_0,\ldots,a_{n-1},b,z_0,z_1,\ldots)\,\dd\nu_\phi(z).
\]
Then
\[
\begin{aligned}
\int (T_{n+1}H-T_nH)\,\dd\nu
&=\int \sum_{b\in S}\Phi_n(y_0,\ldots,y_{n-1},b) \\
&\qquad\qquad\cdot\bigl(Q_n(b\mid y_0,\ldots,y_{n-1})-p_b\bigr)\,\dd\nu.
\end{aligned}
\]
For each fixed prefix \(a\), the function \(b\mapsto\Phi_n(a,b)\) has oscillation at most \(|H|_\theta\theta^n\). Since \(Q_n(\cdot\mid a)-p\) has total mass zero,
\[
\left|\sum_{b\in S}\Phi_n(a,b)\bigl(Q_n(b\mid a)-p_b\bigr)\right|
\le |H|_\theta\theta^n \TV(Q_n(\cdot\mid a),p).
\]
Pinsker's inequality and Cauchy--Schwarz therefore give
\[
\left|\int (T_{n+1}H-T_nH)\,\dd\nu\right|
\le
\frac{|H|_\theta\theta^n}{\sqrt2}
\left(\int D(Q_n(\cdot\mid y_0,\ldots,y_{n-1})\|p)\,\dd\nu\right)^{1/2}.
\]

It remains to bound the averaged conditional relative entropy. Since \(\int \tau\,\dd\nu<\infty\), the conditional first moment of \(Q_n(\cdot\mid y_0,\ldots,y_{n-1})\) is finite for \(\nu\)-almost every prefix. Applying \cref{lem:finite-entropy} conditionally and integrating gives
\[
\begin{aligned}
\int D(Q_n(\cdot\mid y_0,\ldots,y_{n-1})\|p)\,\dd\nu
&=P\int \tau(y_n)\,\dd\nu
-H_\nu(y_n\mid \mathcal F_n)\\
&\qquad-\int\psi(y_n)\,\dd\nu.
\end{aligned}
\]
By stationarity, \(\int \tau(y_n)\,\dd\nu=\int\tau\,\dd\nu\) and \(\int\psi(y_n)\,\dd\nu=\int\psi(y_0)\,\dd\nu\). Also
\[
H_\nu(y_n\mid \mathcal F_n)\ge h_\nu(\sigma),
\]
because the conditional entropies decrease to the entropy rate; see \cite{walters1982,gray2011}. Therefore
\[
\int D(Q_n(\cdot\mid y_0,\ldots,y_{n-1})\|p)\,\dd\nu
\le \Delta_Y(\nu).
\]
Consequently,
\[
\left|\int (T_{n+1}H-T_nH)\,\dd\nu\right|
\le
\frac{|H|_\theta\theta^n}{\sqrt2}\sqrt{\Delta_Y(\nu)}.
\]
Summing over \(n\ge0\) proves the theorem.\end{proof}

\section{Transfer back to the \texorpdfstring{\(S\)}{S}-gap shift}\label{sec:transfer}

From now on, write
\[
\Delta_\phi(\mu)=P-\left(h_\mu(\sigma)+\int\phi\,\dd\mu\right),
\qquad \mu\in\M_\sigma(\X).
\]
By \cref{prop:equilibrium}, \(\Delta_\phi(\mu)\ge0\).

\subsection{The bounded remainder of an induced observable}

For \(f\in C^\theta(\X)\), set
\[
c_f=f(0^\infty),
\qquad
G_f(y)=F_f(y)-c_f\tau(y).
\]

\begin{lemma}\label{lem:Gf}
There is a constant \(C_\theta>0\), depending only on \(\theta\), such that for every \(f\in C^\theta(\X)\),
\[
\norm{G_f}_\infty\le C_\theta |f|_\theta,
\qquad
|G_f|_\theta\le C_\theta |f|_\theta.
\]
In particular, \(\norm{G_f}_\theta\le C_\theta\norm{f}_\theta\).
\end{lemma}

\begin{proof}
Let \(r=\tau(y)=y_0+1\).  The point \(\pi(y)\) may differ from \(0^\infty\) at the first coordinate, and for \(1\le j\le r-1\), the point \(\sigma^j\pi(y)\) begins with \(r-j\) zeros.  Hence
\[
|G_f(y)|
\le |f|_\theta\left(1+\sum_{j=1}^{r-1}\theta^{r-j}\right)
\le \frac{|f|_\theta}{1-\theta}.
\]

Now let \(y,z\in\Y\) and suppose that \(t(y,z)=n\ge1\).  Then \(y_i=z_i\) for \(0\le i<n\).  In particular \(\tau(y)=\tau(z)\), and \(\pi(y)\) and \(\pi(z)\) agree for at least
\[
L_n=\sum_{i=0}^{n-1}(y_i+1)\ge n
\]
coordinates.  For \(0\le j<\tau(y)\),
\[
|f(\sigma^j\pi(y))-f(\sigma^j\pi(z))|
\le |f|_\theta\theta^{L_n-j}.
\]
Since \(L_n\ge \tau(y)+n-1\), summing in \(j\) yields
\[
|G_f(y)-G_f(z)|
\le \frac{|f|_\theta}{1-\theta}\theta^n.
\]
The case \(n=0\) follows from the sup norm bound.\end{proof}

If \(\mu(\{0^\infty\})=0\), let \(\nu\) be its induced measure and put \(m=\int\tau\,\dd\nu\).  The inducing formulas give
\begin{equation}\label{eq:f-ratio}
\int f\,\dd\mu
=f(0^\infty)+\frac1m\int G_f\,\dd\nu.
\end{equation}
Similarly,
\begin{equation}\label{eq:f-ratio-equilibrium}
\int f\,\dd\mu_\phi
=f(0^\infty)+\frac1{\bar\tau_\phi}\int G_f\,\dd\nu_\phi.
\end{equation}

\subsection{Defect identity for measures without the fixed-point atom}

\begin{lemma}\label{lem:defect-identity}
Let \(\mu\in\M_\sigma(\X)\) satisfy \(\mu(\{0^\infty\})=0\), and let \(\nu\) be its induced measure.  If \(m=\int\tau\,\dd\nu\), then
\[
\Delta_Y(\nu)=m\Delta_\phi(\mu).
\]
\end{lemma}

\begin{proof}
By \cref{lem:kac-abramov},
\[
h_\mu(\sigma)=\frac{h_\nu(\sigma)}m,
\qquad
\int\phi\,\dd\mu=\frac1m\int\psi(y_0)\,\dd\nu.
\]
Multiplying \(\Delta_\phi(\mu)\) by \(m\) gives exactly \(\Delta_Y(\nu)\).\end{proof}

\subsection{Controlling the return-time average}

We need one elementary estimate on the original system.  Let
\[
B_\phi=\frac{|\phi|_\theta}{1-\theta}.
\]

\begin{lemma}\label{lem:phi-control}
For every \(\mu\in\M_\sigma(\X)\),
\[
\int\phi\,\dd\mu
\le
\phi(0^\infty)+B_\phi\mu([1]).
\]
\end{lemma}

\begin{proof}
Let \(N(x)=\min\{k\ge0:x_k=1\}\), with \(N(0^\infty)=\infty\).  Then
\[
d_\theta(x,0^\infty)\le \theta^{N(x)},
\]
where \(\theta^\infty=0\).  Hence
\[
\phi(x)\le \phi(0^\infty)+|\phi|_\theta\theta^{N(x)}.
\]
Since \(\{N=k\}\subset\{x_k=1\}=\sigma^{-k}[1]\), invariance gives
\[
\int\theta^{N(x)}\,\dd\mu(x)
\le \sum_{k\ge0}\theta^k\mu(\sigma^{-k}[1])
=\frac{\mu([1])}{1-\theta}.
\]
The result follows.\end{proof}

\begin{lemma}\label{lem:roof-control}
There are constants \(\delta_0>0\), \(C_\tau>0\), and \(M_0>0\), depending only on \(S,\phi\) and \(\theta\), with the following property.  If \(\mu\in\M_\sigma(\X)\), \(\mu(\{0^\infty\})=0\), and \(\Delta_\phi(\mu)\le\delta_0\), then for the induced measure \(\nu\) of \(\mu\),
\[
\left|\int\tau\,\dd\nu-\bar\tau_\phi\right|
\le C_\tau\sqrt{\Delta_\phi(\mu)}
\]
and
\[
\int\tau\,\dd\nu\le M_0.
\]
\end{lemma}

\begin{proof}
Write
\[
m=\int\tau\,\dd\nu,
\qquad
p_\mu=\mu([1]).
\]
Since \(\mu(\{0^\infty\})=0\), Kac's formula gives \(m=1/p_\mu\).

We first show that \(m\) is uniformly bounded when the pressure defect is small.  Let
\[
\gamma=P-\phi(0^\infty)>0.
\]
For a stationary binary process with \(\mu([1])=p_\mu\),
\[
h_\mu(\sigma)\le H_{\mathrm{bin}}(p_\mu),
\]
where \(H_{\mathrm{bin}}(u)=-u\log u-(1-u)\log(1-u)\).  By \cref{lem:phi-control},
\[
h_\mu(\sigma)+\int\phi\,\dd\mu
\le
\phi(0^\infty)+H_{\mathrm{bin}}(p_\mu)+B_\phi p_\mu.
\]
Choose \(p_*>0\) so small that
\[
H_{\mathrm{bin}}(u)+B_\phi u<\frac\gamma2
\qquad (0\le u<p_*).
\]
If \(\Delta_\phi(\mu)\le\gamma/2\), then the free energy is at least \(\phi(0^\infty)+\gamma/2\), so necessarily \(p_\mu\ge p_*\).  Thus
\begin{equation}\label{eq:m-bound-small}
m\le M_*:=p_*^{-1}
\end{equation}
whenever \(\Delta_\phi(\mu)\le\gamma/2\).

Let \(q\) be the one-coordinate marginal of \(\nu\).  Then
\[
m=\sum_{n\in S}q_n(n+1).
\]
Since \(H(q)\ge h_\nu(\sigma)\), \cref{lem:finite-entropy} and \cref{lem:defect-identity} give
\begin{equation}\label{eq:marginal-D-bound}
D(q\|p)
\le
\Delta_Y(\nu)
=m\Delta_\phi(\mu).
\end{equation}
If \(\Delta_\phi(\mu)\le\gamma/2\), then \eqref{eq:m-bound-small} gives
\begin{equation}\label{eq:D-small-bound}
D(q\|p)
\le M_*\Delta_\phi(\mu).
\end{equation}

It remains to convert relative entropy into control of the first moment.  By \cref{lem:psi-tail}, there is \(\eta>0\) such that \(\sum_np_ne^{\eta(n+1)}<\infty\).  Define
\[
M(t)=\sum_{n\in S}p_n e^{t((n+1)-\bar\tau_\phi)},
\qquad
\Lambda(t)=\log M(t).
\]
Then \(M(t)<\infty\) for \(|t|\le \eta/2\), and \(M\) is twice continuously differentiable on this interval.  Since \(M(0)=1\) and \(M'(0)=0\), there are \(t_0>0\) and \(A>0\), depending only on \(S,\phi,\theta\), such that
\begin{equation}\label{eq:lambda-quadratic}
\Lambda(t)\le At^2
\qquad (|t|\le t_0).
\end{equation}
For \(|t|\le t_0\), set
\[
p_n^{(t)}=e^{-\Lambda(t)}p_n e^{t((n+1)-\bar\tau_\phi)}.
\]
Then \((p_n^{(t)})\) is a probability vector.  Since \(D(q\|p^{(t)})\ge0\),
\[
t(m-\bar\tau_\phi)
\le D(q\|p)+\Lambda(t)
\]
for \(0<t\le t_0\), and applying the same inequality with \(-t\) gives
\[
|m-\bar\tau_\phi|
\le \frac{D(q\|p)}t+At
\qquad (0<t\le t_0).
\]
Choose
\[
\delta_0=\min\left\{\frac\gamma2,\frac{t_0^2}{M_*}\right\}.
\]
If \(\Delta_\phi(\mu)\le\delta_0\), then \eqref{eq:D-small-bound} gives \(D(q\|p)\le t_0^2\).  Taking \(t=\sqrt{D(q\|p)}\) when \(D(q\|p)>0\), and arguing trivially when \(D(q\|p)=0\), we obtain
\[
|m-\bar\tau_\phi|
\le (1+A)\sqrt{D(q\|p)}
\le (1+A)\sqrt{M_*}\sqrt{\Delta_\phi(\mu)}.
\]
This proves the first estimate.  The second estimate follows from \eqref{eq:m-bound-small}; set \(M_0=M_*\).\end{proof}

\subsection{The estimate for measures without the fixed-point atom}

\begin{proposition}\label{prop:no-atom-effective}
There is a constant \(C_0>0\), depending only on \(S,\phi\) and \(\theta\), such that for every \(\mu\in\M_\sigma(\X)\) satisfying \(\mu(\{0^\infty\})=0\) and every \(f\in C^\theta(\X)\),
\[
\left|\int f\,\dd\mu-\int f\,\dd\mu_\phi\right|
\le
C_0\norm{f}_\theta\sqrt{\Delta_\phi(\mu)}.
\]
\end{proposition}

\begin{proof}
Let \(\delta_0,C_\tau,M_0\) be as in \cref{lem:roof-control}.  If \(\Delta_\phi(\mu)\ge\delta_0\), then
\[
\left|\int f\,\dd\mu-\int f\,\dd\mu_\phi\right|
\le 2\norm{f}_\infty
\le 2\delta_0^{-1/2}\norm{f}_\theta\sqrt{\Delta_\phi(\mu)}.
\]

Assume now that \(\Delta_\phi(\mu)<\delta_0\).  Let \(\nu\) be the induced measure of \(\mu\), and set \(m=\int\tau\,\dd\nu\).  By \eqref{eq:f-ratio} and \eqref{eq:f-ratio-equilibrium},
\[
\left|\int f\,\dd\mu-\int f\,\dd\mu_\phi\right|
\le
\left|\frac1m\int G_f\,\dd\nu-\frac1{\bar\tau_\phi}\int G_f\,\dd\nu_\phi\right|.
\]
Thus
\[
\left|\int f\,\dd\mu-\int f\,\dd\mu_\phi\right|
\le
\frac1m\left|\int G_f\,\dd\nu-\int G_f\,\dd\nu_\phi\right|
+
\norm{G_f}_\infty\left|\frac1m-\frac1{\bar\tau_\phi}\right|.
\]
Since \(m\ge1\), \cref{thm:induced-gap,lem:defect-identity,lem:Gf,lem:roof-control} give
\[
\begin{aligned}
\left|\int G_f\,\dd\nu-\int G_f\,\dd\nu_\phi\right|
&\le \frac{|G_f|_\theta}{\sqrt2(1-\theta)}\sqrt{\Delta_Y(\nu)}\\
&= \frac{|G_f|_\theta}{\sqrt2(1-\theta)}\sqrt{m\Delta_\phi(\mu)}\\
&\le C_1\norm{f}_\theta\sqrt{\Delta_\phi(\mu)},
\end{aligned}
\]
where \(m\le M_0\) has been used in the last line.  Also,
\[
\left|\frac1m-\frac1{\bar\tau_\phi}\right|
=\frac{|m-\bar\tau_\phi|}{m\bar\tau_\phi}
\le \frac{C_\tau}{\bar\tau_\phi}\sqrt{\Delta_\phi(\mu)}.
\]
Together with \(\norm{G_f}_\infty\le C_\theta\norm{f}_\theta\), this proves the desired estimate in the small-defect case.\end{proof}

\subsection{Proof of the main theorem}

\begin{proof}[Proof of \cref{thm:main}]
The pressure identity and uniqueness of \(\mu_\phi\) were proved in \cref{prop:equilibrium}.  It remains only to prove the effective estimate for an arbitrary invariant measure \(\mu\).

Set
\[
a=\mu(\{0^\infty\}),
\qquad
\beta=1-a,
\qquad
\gamma=P-\phi(0^\infty)>0.
\]
If \(\beta=0\), then \(\mu=\delta_{0^\infty}\) and \(\Delta_\phi(\mu)=\gamma\), so
\[
\left|\int f\,\dd\mu-\int f\,\dd\mu_\phi\right|
\le 2\gamma^{-1/2}\norm{f}_\theta\sqrt{\Delta_\phi(\mu)}.
\]
Assume \(\beta>0\), and let \(\widetilde\mu\) be as in \cref{lem:atom-splitting}.  Then \(\widetilde\mu(\{0^\infty\})=0\), and \cref{prop:no-atom-effective} gives
\[
\left|\int f\,\dd\widetilde\mu-\int f\,\dd\mu_\phi\right|
\le C_0\norm{f}_\theta\sqrt{\Delta_\phi(\widetilde\mu)}.
\]
By \eqref{eq:defect-splitting},
\[
\Delta_\phi(\mu)=a\gamma+\beta\Delta_\phi(\widetilde\mu).
\]
Therefore
\[
a\le \frac{\Delta_\phi(\mu)}\gamma,
\qquad
\beta\sqrt{\Delta_\phi(\widetilde\mu)}
=\sqrt{\beta(\beta\Delta_\phi(\widetilde\mu))}
\le \sqrt{\Delta_\phi(\mu)}.
\]
It follows that
\[
\begin{aligned}
\left|\int f\,\dd\mu-\int f\,\dd\mu_\phi\right|
&\le
\beta\left|\int f\,\dd\widetilde\mu-\int f\,\dd\mu_\phi\right|
+a\left|f(0^\infty)-\int f\,\dd\mu_\phi\right|\\
&\le
C_0\norm{f}_\theta\sqrt{\Delta_\phi(\mu)}
+2\norm{f}_\infty\frac{\Delta_\phi(\mu)}\gamma.
\end{aligned}
\]
Finally, since \(h_\mu(\sigma)\ge0\) and \(\int\phi\,\dd\mu\ge-\norm{\phi}_\infty\),
\[
\Delta_\phi(\mu)
\le D_{\max}:=P+\norm{\phi}_\infty.
\]
Thus
\[
2\norm{f}_\infty\frac{\Delta_\phi(\mu)}\gamma
\le
\frac{2\sqrt{D_{\max}}}{\gamma}\norm{f}_\theta\sqrt{\Delta_\phi(\mu)}.
\]
Combining the last estimates proves the theorem.\end{proof}

\section{Cylinder-avoiding sets and dimension}\label{sec:applications}

Throughout this section we assume the hypotheses of \cref{thm:main}. In particular \(P=P_{\mathrm{top}}(\X,\sigma,\phi)\), and \(\mu_\phi\) denotes the unique equilibrium state.

Let \(C=[w]\subset\X\) be a cylinder determined by an admissible word \(w\) of length \(\ell\ge1\).  Define
\[
\mathcal E_C=\{x\in\X:\sigma^k x\notin C \text{ for every } k\ge0\}.
\]
This is a compact forward-invariant set.  If \(\mu\in\M_\sigma(\X)\) and \(\mu(\mathcal E_C)=1\), then \(\mu(C)=0\).

\begin{lemma}\label{lem:cylinder-holder}
The cylinder indicator \(\1_C\) belongs to \(C^\theta(\X)\), and
\[
\norm{\1_C}_\theta\le 1+\theta^{-(\ell-1)}.
\]
\end{lemma}

\begin{proof}
The sup norm is \(1\).  If \(\1_C(x)\ne\1_C(y)\), then \(x\) and \(y\) differ in one of the first \(\ell\) coordinates.  Hence \(d_\theta(x,y)\ge\theta^{\ell-1}\), and the H\"older seminorm is at most \(\theta^{-(\ell-1)}\).\end{proof}

\begin{proposition}\label{prop:cylinder-pressure-drop}
Assume \(\mu_\phi(C)>0\).  Then every invariant measure \(\mu\) supported on \(\mathcal E_C\) satisfies
\[
P-\left(h_\mu(\sigma)+\int\phi\,\dd\mu\right)
\ge
\kappa_C,
\]
where, if \(C_{\mathrm{eff}}\) denotes the constant from \cref{thm:main},
\[
\kappa_C=
\frac{\mu_\phi(C)^2}{C_{\mathrm{eff}}^2\norm{\1_C}_\theta^2}.
\]
Consequently,
\[
P_{\mathrm{top}}(\mathcal E_C,\sigma,\phi)
\le P-\kappa_C,
\]
where the pressure on \(\mathcal E_C\) is computed using invariant measures supported on \(\mathcal E_C\).
\end{proposition}

\begin{proof}
If \(\mu(\mathcal E_C)=1\), then \(\mu(C)=0\).  Applying \cref{thm:main} to \(f=\1_C\), we get
\[
\mu_\phi(C)
=
\left|\int\1_C\,\dd\mu-\int\1_C\,\dd\mu_\phi\right|
\le
C_{\mathrm{eff}}\norm{\1_C}_\theta\sqrt{\Delta_\phi(\mu)}.
\]
This gives the lower bound on \(\Delta_\phi(\mu)\).  Taking the supremum over all invariant measures supported on \(\mathcal E_C\) gives the pressure estimate.\end{proof}

In the entropy case \(\phi=0\), the pressure separation condition is \(h_{\mathrm{top}}(\X)>0\). Under this assumption \(\mu_0:=\mu_\phi\) is the unique measure of maximal entropy. Writing \(h=h_{\mathrm{top}}(\X)\), \cref{prop:cylinder-pressure-drop} gives
\begin{equation}\label{eq:entropy-drop}
h_{\mathrm{top}}(\mathcal E_C)
\le h_{\mathrm{top}}(\X)-\kappa_C
\end{equation}
whenever \(\mu_0(C)>0\) and \(\mathcal E_C\) is nonempty. If \(\mathcal E_C\) is empty, the pressure and entropy estimates are trivial; in the Hausdorff-dimension statement below we use a slightly smaller positive gap to cover both cases without relying on a convention for the entropy of the empty set.

For the dimension estimate we use cylinder-covering entropy.  This is the symbolic form of a standard entropy-dimension estimate; see, for example, Pesin's text on dimension theory \cite{pesin1997dimension}.  For compact \(Z\subset\X\), let \(N_n(Z)\) be the least number of length-\(n\) cylinders needed to cover \(Z\), and set
\[
h_{\mathrm{cyl}}(Z)=\limsup_{n\to\infty}\frac1n\log N_n(Z).
\]
For compact shift-invariant subsets of a one-sided subshift, this agrees with the usual topological entropy.

\begin{lemma}\label{lem:entropy-dimension}
For every compact \(Z\subset\X\),
\[
\dim_H(Z,d_\theta)
\le
\frac{h_{\mathrm{cyl}}(Z)}{-\log\theta}.
\]
\end{lemma}

\begin{proof}
Let \(h_Z=h_{\mathrm{cyl}}(Z)\), and choose \(s>h_Z/(-\log\theta)\). For every small \(\varepsilon>0\) and all large \(n\), the set \(Z\) can be covered by at most \(\exp(n(h_Z+\varepsilon))\) cylinders of length \(n\). Each such cylinder has diameter at most \(\theta^n\). Thus the \(s\)-dimensional Hausdorff content is bounded by
\[
\exp(n(h_Z+\varepsilon))\theta^{ns}.
\]
Choose \(\varepsilon>0\) so that \(h_Z+\varepsilon+s\log\theta<0\), and let \(n\to\infty\). This gives \(\mathcal H^s(Z)=0\).\end{proof}

\begin{corollary}\label{cor:dimension-drop}
Assume \(\phi=0\) and \(h_{\mathrm{top}}(\X)>0\). Let \(C\subset\X\) be a cylinder with \(\mu_0(C)>0\), and set
\[
\widehat\kappa_C=\min\left\{\kappa_C,\frac12h_{\mathrm{top}}(\X)\right\}.
\]
Then \(\widehat\kappa_C>0\) and
\[
\dim_H(\mathcal E_C,d_\theta)
\le
\frac{h_{\mathrm{top}}(\X)-\widehat\kappa_C}{-\log\theta}
<
\frac{h_{\mathrm{top}}(\X)}{-\log\theta}.
\]
\end{corollary}

\begin{proof}
If \(\mathcal E_C=\emptyset\), then \(\dim_H(\mathcal E_C,d_\theta)=0\), and the displayed bound is immediate because \(\widehat\kappa_C\le h_{\mathrm{top}}(\X)/2\). If \(\mathcal E_C\) is nonempty, then it is compact and shift-invariant, so \(h_{\mathrm{cyl}}(\mathcal E_C)=h_{\mathrm{top}}(\mathcal E_C)\). By \eqref{eq:entropy-drop},
\[
h_{\mathrm{cyl}}(\mathcal E_C)\le h_{\mathrm{top}}(\X)-\kappa_C\le h_{\mathrm{top}}(\X)-\widehat\kappa_C.
\]
The claim follows from \cref{lem:entropy-dimension}.\end{proof}

\begin{example}
Let \(a\in S\), and let \(C_a=[10^a1]\). Under the entropy assumptions of \cref{cor:dimension-drop}, the renewal equation is
\[
\sum_{n\in S}e^{-h(n+1)}=1,
\]
and the induced Bernoulli weights of \(\mu_0\) are \(p_n=e^{-h(n+1)}\).  By \cref{cor:block-weights},
\[
\mu_0(C_a)=\mu_0([1])e^{-h(a+1)}.
\]
Since \(C_a\) has defining word length \(a+2\), \cref{lem:cylinder-holder} gives
\[
\norm{\1_{C_a}}_\theta\le 1+\theta^{-(a+1)}.
\]
Thus the set of points whose orbits never enter \([10^a1]\) has entropy at most \(h-\widehat\kappa_{C_a}\), where one may take
\[
\widehat\kappa_{C_a}
=
\min\left\{
\frac h2,
\frac{\mu_0([1])^2e^{-2h(a+1)}}{C_{\mathrm{eff}}^2(1+\theta^{-(a+1)})^2}
\right\}.
\]
\end{example}

\section{Final remarks}

The main result gives an effective intrinsic ergodicity for renewal-type potentials on one-sided \(S\)-gap shifts.  The estimate is proved on the original compact symbolic system, not only on the induced countable renewal shift.  This distinction is important: the passage from the induced system back to \(\X\) changes both observables and pressure defects.

The proof has two main ingredients.  First, the square-root pressure-gap inequality on the induced renewal full shift is obtained directly from conditional relative entropy and a telescoping decomposition of H\"older observables.  In the renewal-type setting, the induced potential is one-symbol and the induced equilibrium measure is Bernoulli, so this part of the argument does not require spectral theory or the general theory of strongly positively recurrent countable Markov shifts.

Second, the transfer from the induced system to \(\X\) requires more than a formal lifting argument.  If \(f\in C^\theta(\X)\), then the induced observable
\[
F_f(y)=\sum_{j=0}^{\tau(y)-1}f(\sigma^j\pi(y))
\]
is generally unbounded because of the roof term.  The decomposition
\[
F_f=f(0^\infty)\tau+G_f
\]
separates this unbounded part from a bounded H\"older remainder.  The return-time average is then controlled by a pressure-gap estimate, using the exponential tail of the induced Bernoulli weights.  The atom-splitting step is also essential: Kac--Abramov formulas are applied only after removing the possible fixed-point atom at \(0^\infty\).

The cylinder-avoidance application shows one concrete consequence of the effective estimate.  If an orbit avoids a fixed cylinder of positive equilibrium measure, then the corresponding exceptional set has a uniform pressure drop.  In the entropy case this gives a strict entropy drop and hence a strict Hausdorff-dimension upper bound in the symbolic metric.

Two natural questions remain.  One is whether the renewal-type assumption can be replaced by a broader class of H\"older potentials for which the induced equilibrium measure is no longer Bernoulli.  Another is whether the same atom-splitting and return-time control can be adapted to other coded symbolic systems with a renewal-type inducing scheme.

\bibliographystyle{amsplain}
\bibliography{mybib}

@book{aaronson1997,
  author    = {Aaronson, Jon},
  title     = {An Introduction to Infinite Ergodic Theory},
  series    = {Mathematical Surveys and Monographs},
  volume    = {50},
  publisher = {American Mathematical Society},
  address   = {Providence, RI},
  year      = {1997},
  doi       = {10.1090/surv/050}
}

@article{abramov1959,
  author  = {Abramov, L. M.},
  title   = {The entropy of a derived automorphism},
  journal = {Doklady Akademii Nauk SSSR},
  volume  = {128},
  pages   = {647--650},
  year    = {1959},
  note    = {In Russian; English translation in American Mathematical Society Translations, Series 2, 49 (1966), 162--166}
}

@book{baladi2000positive,
  author    = {Baladi, Viviane},
  title     = {Positive Transfer Operators and Decay of Correlations},
  series    = {Advanced Series in Nonlinear Dynamics},
  volume    = {16},
  publisher = {World Scientific},
  address   = {Singapore},
  year      = {2000}
}

@inproceedings{baladi2014linear,
  author    = {Baladi, Viviane},
  title     = {Linear response, or else},
  booktitle = {Proceedings of the International Congress of Mathematicians, Seoul 2014},
  volume    = {III},
  pages     = {525--545},
  publisher = {Kyung Moon Sa},
  address   = {Seoul},
  year      = {2014}
}

@book{bowen1975equilibrium,
  author    = {Bowen, Rufus},
  title     = {Equilibrium States and the Ergodic Theory of {Anosov} Diffeomorphisms},
  series    = {Lecture Notes in Mathematics},
  volume    = {470},
  publisher = {Springer-Verlag},
  address   = {Berlin},
  year      = {1975}
}

@article{climenhaga2012intrinsic,
  author  = {Climenhaga, Vaughn and Thompson, Daniel J.},
  title   = {Intrinsic ergodicity beyond specification: {$\beta$}-shifts, {$S$}-gap shifts, and their factors},
  journal = {Israel Journal of Mathematics},
  volume  = {192},
  number  = {2},
  pages   = {785--817},
  year    = {2012},
  doi     = {10.1007/s11856-012-0056-0}
}

@article{climenhaga2013bowen,
  author  = {Climenhaga, Vaughn and Thompson, Daniel J.},
  title   = {Equilibrium states beyond specification and the {Bowen} property},
  journal = {Journal of the London Mathematical Society},
  volume  = {87},
  number  = {2},
  pages   = {401--427},
  year    = {2013},
  doi     = {10.1112/jlms/jds054}
}

@article{climenhaga2016unique,
  author  = {Climenhaga, Vaughn and Thompson, Daniel J.},
  title   = {Unique equilibrium states for flows and homeomorphisms with non-uniform structure},
  journal = {Advances in Mathematics},
  volume  = {303},
  pages   = {745--799},
  year    = {2016},
  doi     = {10.1016/j.aim.2016.08.013}
}

@article{dastjerdi2012,
  author  = {Dastjerdi, Dawoud Ahmadi and Jangjoo, Somaye},
  title   = {Dynamics and topology of {$S$}-gap shifts},
  journal = {Topology and its Applications},
  volume  = {159},
  number  = {10--11},
  pages   = {2654--2661},
  year    = {2012},
  doi     = {10.1016/j.topol.2012.04.002}
}

@article{gouezel2006banach,
  author  = {Gou\"ezel, S\'ebastien and Liverani, Carlangelo},
  title   = {Banach spaces adapted to {Anosov} systems},
  journal = {Ergodic Theory and Dynamical Systems},
  volume  = {26},
  number  = {1},
  pages   = {189--217},
  year    = {2006},
  doi     = {10.1017/S0143385705000374}
}

@book{gray2011,
  author    = {Gray, Robert M.},
  title     = {Entropy and Information Theory},
  edition   = {2},
  publisher = {Springer},
  address   = {New York},
  year      = {2011}
}

@article{kac1947,
  author  = {Kac, Mark},
  title   = {On the notion of recurrence in discrete stochastic processes},
  journal = {Bulletin of the American Mathematical Society},
  volume  = {53},
  pages   = {1002--1010},
  year    = {1947},
  doi     = {10.1090/S0002-9904-1947-08927-8}
}

@article{kadyrov2015effective,
  author  = {Kadyrov, Shirali},
  title   = {Effective uniqueness of {Parry} measure and exceptional sets in ergodic theory},
  journal = {Monatshefte f\"ur Mathematik},
  volume  = {178},
  number  = {2},
  pages   = {237--249},
  year    = {2015},
  doi     = {10.1007/s00605-014-0712-8}
}

@book{keller1998equilibrium,
  author    = {Keller, Gerhard},
  title     = {Equilibrium States in Ergodic Theory},
  series    = {London Mathematical Society Student Texts},
  volume    = {42},
  publisher = {Cambridge University Press},
  address   = {Cambridge},
  year      = {1998}
}

@article{keller1999stability,
  author  = {Keller, Gerhard and Liverani, Carlangelo},
  title   = {Stability of the spectrum for transfer operators},
  journal = {Annali della Scuola Normale Superiore di Pisa, Classe di Scienze},
  volume  = {28},
  number  = {1},
  pages   = {141--152},
  year    = {1999}
}

@book{lindmarcus2021,
  author    = {Lind, Douglas and Marcus, Brian},
  title     = {An Introduction to Symbolic Dynamics and Coding},
  edition   = {2},
  publisher = {Cambridge University Press},
  address   = {Cambridge},
  year      = {2021},
  doi       = {10.1017/9781108899727}
}

@book{pesin1997dimension,
  author    = {Pesin, Yakov B.},
  title     = {Dimension Theory in Dynamical Systems: Contemporary Views and Applications},
  series    = {Chicago Lectures in Mathematics},
  publisher = {University of Chicago Press},
  address   = {Chicago},
  year      = {1997}
}

@article{pollicott2026,
  title={Effective intrinsic ergodicity for expanding interval maps},
  author={Pollicott, Mark},
  journal={arXiv preprint arXiv:2606.12081},
  year={2026}
}

@article{ruhr2021pressure,
  author  = {R\"uhr, Ren\'e},
  title   = {Pressure inequalities for {Gibbs} measures of countable {Markov} shifts},
  journal = {Dynamical Systems},
  volume  = {36},
  number  = {2},
  pages   = {332--339},
  year    = {2021},
  doi     = {10.1080/14689367.2020.1856422}
}

@article{ruhr2022effective,
  author  = {R\"uhr, Ren\'e and Sarig, Omri},
  title   = {Effective intrinsic ergodicity for countable state {Markov} shifts},
  journal = {Israel Journal of Mathematics},
  volume  = {251},
  number  = {2},
  pages   = {679--735},
  year    = {2022},
  doi     = {10.1007/s11856-022-2385-x}
}

@book{ruelle2004thermodynamic,
  author    = {Ruelle, David},
  title     = {Thermodynamic Formalism: The Mathematical Structures of Equilibrium Statistical Mechanics},
  publisher = {Cambridge University Press},
  address   = {Cambridge},
  edition   = {2},
  year      = {2004}
}

@article{sarig1999thermodynamic,
  author  = {Sarig, Omri M.},
  title   = {Thermodynamic formalism for countable {Markov} shifts},
  journal = {Ergodic Theory and Dynamical Systems},
  volume  = {19},
  number  = {6},
  pages   = {1565--1593},
  year    = {1999},
  doi     = {10.1017/S0143385799171001}
}

@article{sarig2001thermodynamic,
  author  = {Sarig, Omri M.},
  title   = {Thermodynamic formalism for null recurrent potentials},
  journal = {Israel Journal of Mathematics},
  volume  = {121},
  pages   = {285--311},
  year    = {2001},
  doi     = {10.1007/BF02804979}
}

@article{walters1975variational,
  author  = {Walters, Peter},
  title   = {A variational principle for the pressure of continuous transformations},
  journal = {American Journal of Mathematics},
  volume  = {97},
  number  = {4},
  pages   = {937--971},
  year    = {1975},
  doi     = {10.2307/2373688}
}

@book{walters1982,
  author    = {Walters, Peter},
  title     = {An Introduction to Ergodic Theory},
  series    = {Graduate Texts in Mathematics},
  volume    = {79},
  publisher = {Springer-Verlag},
  address   = {New York},
  year      = {1982}
}

\end{document}